\documentclass[10pt]{amsart}

\usepackage{amsmath,amsthm,mathrsfs}
\usepackage{amssymb,amsfonts}
\usepackage{url,relsize,xcolor}
\usepackage{tikz,tikz-cd}
\usepackage{enumitem}
\usepackage{xr}
\usepackage{nicefrac,xspace,mathtools,marvosym,verbatim}

\newtheorem{theorem}{Theorem}[section]
\newtheorem*{theorem1}{Theorem}
\newtheorem*{theorem2}{Theorem}
\newtheorem{proposition}[theorem]{Proposition}
\newtheorem{lemma}[theorem]{Lemma}
\newtheorem{question}[theorem]{Question}
\newtheorem*{claim}{Claim}

\newtheorem{corollary}[theorem]{Corollary}

\newtheorem{observation}[theorem]{Observation}

\theoremstyle{definition}

\newcommand{\Q}{\mathbb{Q}}
\newcommand{\R}{\mathbb{R}}
\newcommand{\N}{\mathbb{N}}

\newcommand{\CC}{\mathbb{C}}
\renewcommand{\AA}{\mathbb{A}}

\newcommand{\F}{\mathcal{F}}

\newcommand{\I}{\mathcal{I}}

\newcommand{\M}{\mathcal{M}}
\newcommand{\NN}{\mathcal{N}}

\newcommand{\explicitSet}[1]{\left\lbrace #1 \right\rbrace}
\newcommand{\brackets}[1]{\left\langle #1 \right\rangle}
\newcommand{\set}[2]{\explicitSet{#1 \colon #2}}
\newcommand{\seq}[2]{\brackets{#1 \colon #2}}

\newcommand{\tr}[1]{\langle\hspace{-.8mm}\langle #1 \rangle\hspace{-.8mm}\rangle}
\newcommand{\norm}[1]{|\hspace{-.15mm}| #1 |\hspace{-.15mm}|}
\newcommand{\bignorm}[1]{\big|\hspace{-.4mm}\big| #1 \big|\hspace{-.4mm}\big|}
\newcommand{\closure}[1]{\overline{#1}}
\renewcommand{\a}{\alpha}

\newcommand{\dlt}{\delta}
\newcommand{\e}{\varepsilon}

\renewcommand{\k}{\kappa}

\newcommand{\w}{\omega}
\newcommand{\0}{\emptyset}

\newcommand{\sub}{\subseteq}
\newcommand{\rest}{\!\restriction\!}

\newcommand{\cf}{\mathrm{cf}}
\newcommand{\ch}{\ensuremath{\mathsf{CH}}\xspace}

\newcommand{\zfc}{\ensuremath{\mathsf{ZFC}}\xspace}
\newcommand{\ssh}{\ensuremath{\mathsf{SSH}}\xspace}
\newcommand{\bdd}{\mathfrak{b}}

\newcommand{\continuum}{\mathfrak{c}}

\newcommand{\bd}{\mathbf{d}}
\newcommand{\be}{\mathbf{e}}
\newcommand{\bx}{\mathbf{x}}
\newcommand{\by}{\mathbf{y}}
\newcommand{\bz}{\mathbf{z}}
\newcommand{\barr}{\mathfrak{y}}

\begin{document}

%%%%%%%%%%%%
\title{Small-dimensional normed barrelled Spaces}
%%%%%%%%%%%%

\author{Will Brian}
\address{
W. R. Brian\\
Department of Mathematics and Statistics\\
University of North Carolina at Charlotte\\
Charlotte, NC, USA}
\email{wbrian.math@gmail.com}
\urladdr{wrbrian.wordpress.com}

\author{Christopher Stuart}
\address{C. E. Stuart\\
Department of Mathematical Sciences\\
New Mexico State University\\
Las Cruces, NM, USA}
\email{cstuart@nmsu.edu}

%%%%%%%%%%%%

\thanks{The first author is supported in part by NSF grant DMS-2154229. }
%%%%%%%%%%%%

%%%%%%%%%%%%
\maketitle
%%%%%%%%%%%%

%%%%%%%%%%%%
\begin{abstract}
We prove that every separable Banach space has a barrelled subspace with algebraic dimension $\mathrm{non}(\mathcal M)$, which denotes the smallest cardinality of a non-meager subset of $\mathbb R$. 
This strengthens a theorem of Sobota. 
More generally, we prove that every Banach space with density character $\kappa$ contains a barrelled subspace with algebraic dimension $\mathrm{cf}[\kappa]^\omega \cdot \mathrm{non}(\mathcal M)$, and in particular it is consistent with $\mathsf{ZFC}$ that every Banach space with density character $<\!\mathfrak{c}$ has a barrelled subspace with dimension $<\!\mathfrak{c}$. 

We also prove that if the dual of a Banach space contains either $c_0$ or $\ell^p$ for some $p \geq 1$, then that space does not have a barrelled subspace with dimension $<\!\mathrm{cov}(\mathcal N)$, which denotes the smallest cardinality of a collection of Lebesgue null sets covering $\mathbb R$. In particular, it is consistent with $\mathsf{ZFC}$ that no classical Banach spaces contain barrelled subspaces with dimension $\mathfrak{b}$. This partly answers a question of S\'anchez Ruiz and Saxon.
\end{abstract}
%%%%%%%%%%%%

%%%%%%%%%%%%
\section{Introduction}
%%%%%%%%%%%%

A classical theorem of Mazur states that every infinite-dimensional Fr\'echet space has dimension at least $\continuum = |\R|$ (see \cite[p. 76]{BPC} or \cite{PT}). In other words, 
$$\continuum \,=\, \min \set{\k}{\text{there is a Fr\'echet space of infinite dimension }\k}.$$
Here, and throughout the paper, the ``dimension'' of a topological vector space space always means its algebraic dimension.

Building in this direction, S\'anchez Ruiz and Saxon authored a series of three papers,  \cite{SSR,SSR2,SSR3}, in which they sought to characterize the smallest cardinality of other kinds of infinite-dimensional topological vector spaces. For example, they prove in \cite{SSR} that there is a metrizable barrelled space of dimension $\bdd$ (their example is a modification of Tweddle's space from \cite{Tweddle}), and that there is no infinite-dimensional metrizable barrelled space of dimension $<\!\bdd$. In other words,  
$$\bdd \,=\, \min \set{\k}{\text{there is a metrizable barrelled space of infinite dimension }\k}.$$
Here $\bdd$ denotes the unbounding number, which is one of the classic cardinal characteristics of the continuum (see \cite{Blass} or \cite{vD}). These are cardinal numbers defined in some combinatorial or topological way, and that lie between $\aleph_1$ and $\continuum$. 

S\'anchez Ruiz and Saxon define another such cardinal:
$$\barr \,=\, \min \set{\k}{\text{there is a normed barrelled space of infinite dimension }\k}.$$ 
Their characterization of $\bdd$ described above shows $\bdd \leq \barr$, and any separable Banach space witnesses $\barr \leq \continuum$. 
They ask in \cite{SSR} and \cite{SSR2} whether $\barr$ might not simply be an alternative characterization of an already-known cardinal characteristic, much like the characterizations of $\bdd$ and $\continuum$ described above. 
The natural candidates at the time were $\barr = \bdd$, the analogous cardinal for metrizable barrelled spaces, or $\barr = \continuum$, like with Mazur's result for Fr\'echet spaces. 
Recent work of Sobota shows that the inequality $\barr < \continuum$ is consistent with \zfc (see Theorem 7.3 and the proof of Proposition 3.2 in \cite{Sobota}), so $\barr$ is not simply a recharacterization of $\continuum$. 

The present paper contains two main theorems. Both are closely tied to the cardinal $\barr$ and the questions raised by S\'anchez Ruiz and Saxon concerning $\barr$, as well as to the work of Sobota and Zdomskyy on Boolean algebras with the Nikodym property \cite{Sobota,SZ1,SZ2,SZ3}. The first of these two results, found as Theorem~\ref{thm:l1+} in Section~\ref{sec:generalization} below, states: 

\begin{theorem1}
If $X$ is an infinite-dimensional Banach space with density character $\k$, then there is a barrelled subspace of $X$ with dimension $\cf[\k]^\w \cdot \mathrm{non}(\M)$.
\end{theorem1}

The cardinal number $\mathrm{non}(\M)$ is defined as the smallest size of a meager subset of $\R$. Clearly $\aleph_1 \leq \mathrm{non}(\M)$ (i.e., non-meager sets are uncountable), and the Baire Category Theorem implies $\mathrm{non}(\mathcal M) \leq \continuum$. The cardinal $\mathrm{non}(\mathcal M)$, known as the \emph{uniformity number of the meager ideal}, is another of the classical cardinal characteristics of the continuum. 
What exactly $\cf[\k]^\w$ means is explained in Section~\ref{sec:generalization}, but let us point out for now that $\cf[\aleph_0]^\w = 1$, which means $\cf[\aleph_0]^\w \cdot \mathrm{non}(\M) = \mathrm{non}(\M)$. 
Thus a special case of the above theorem is: 
\begin{itemize}
\item Every infinite-dimensional separable Banach space has a barrelled subspace with dimension $\mathrm{non}(\M)$.
\end{itemize}
In particular, $\barr \leq \mathrm{non}(\M)$. Combining this with the results of \cite{SSR2}, we get 
$$\aleph_1 \,\leq\, \bdd \,\leq\, \barr \,\leq\, \mathrm{non}(\M) \,\leq\, \continuum.$$
It is consistent with \zfc that $\aleph_1 = \continuum$. This assertion, known as the Continuum Hypothesis and denoted \ch, implies all the above cardinals are equal. 
However, it is also consistent that $\mathrm{non}(\M) < \continuum$. (In fact this is true in Cohen's original model showing that the failure of \ch is consistent with \zfc.) 
Thus our theorem provides an alternative proof of the consistency of $\barr < \continuum$, first proved in \cite{Sobota}. 

Let us point out that our result offers two insights beyond just the consistency of $\barr < \continuum$. First, we identify a classical characteristic, $\mathrm{non}(\M)$, as an upper bound for $\barr$. Second, and perhaps more importantly, our result reveals the ubiquity of small-dimensional barrelled spaces. Not only may such spaces exist, but they may be embedded in every separable Banach space. 
Building on the theme of ubiquity, we prove another corollary to Theorem~\ref{thm:l1+} (see Corollary~\ref{cor:l1+}): 
\begin{itemize}
\item It is consistent with \zfc that $\continuum$ is arbitrarily large, and that every Banach space with density character $<\!\continuum$ has an infinite-dimensional barrelled subspace with dimension $<\!\continuum$. 
\end{itemize}

After proving these theorems about the existence of small-dimensional barrelled spaces, 
in Section~\ref{sec:random} we turn to the related question of just how small their dimension can be. Our main results in this section, summarized in Theorem~\ref{thm:MainCor}, are: 
\begin{theorem2}
If the dual of a Banach space $X$ does not contain $c_0$ or $\ell^p$ for any $p \geq 1$, then $X$ does not have a barrelled subspace with dimension $<\! \mathrm{cov}(\NN)$. 
It is consistent they have no barrelled subspaces of dimension $\bdd$. 
\end{theorem2}
\noindent This result is proved by finding a way to associate members of a Lebesgue probability space to norm-unbounded sequences of functionals on $X$, in such a way that any particular $\bx \in X$ will almost surely not witness the norm-unboundedness of a ``random'' norm-unbounded sequence. 
This ``theorem'' is really two theorems, in the sense that our proof divides cleanly into two separate cases: the $\ell^1$ case, and the case where $X'$ contains a copy of either $c_0$ or of $\ell^p$ for some $p > 1$. 
For the $\ell^1$ case, the ``random'' norm-unbounded sequences of functionals needed in the proof are obtained by taking random walks in the dual of $X$. 

The cardinal $\mathrm{cov}(\NN)$ appearing in this theorem is another cardinal characteristic of the continuum, known as the \emph{covering number of the null ideal}, and defined as 
$$\textstyle \mathrm{cov}(\NN) \,=\,\textstyle \min \set{|\F|}{\F \text{ is a family of Lebesgue null subsets of }\R\text{ and }\bigcup \F \supseteq \R}.$$
Amongst the classical cardinal characteristics mentioned so far, the following inequalities are provable in \zfc:

\begin{center}
\begin{tikzpicture}[scale=.9]

\node at (.1,0) {$\aleph_1$};
\node at (2,.75) {$\mathrm{cov}(\NN)$};
\node at (2,-.75) {$\bdd$};
\node at (4.3,0) {$\mathrm{non}(\M)$};
\node at (5.6,0) {$\leq$};
\node at (6.4,0) {$\continuum$};

\node[rotate=25] at (.9,.5) {\small $\leq$};
\node[rotate=-25] at (.95,-.45) {\small $\leq$};
\node[rotate=25] at (3.05,-.45) {\small $\leq$};
\node[rotate=-25] at (3.1,.5) {\small $\leq$};

\end{tikzpicture}
\end{center}
No further inequalities are provable: it is consistent with \zfc that all these inequalities are strict (even simultaneously), and either of the orderings $\bdd < \mathrm{cov}(\NN)$ and $\mathrm{cov}(\NN) < \bdd$ is consistent (though, of course, not simultaneously). 

In particular, because the inequality $\bdd < \mathrm{cov}(\NN)$ is consistent, our theorem shows it is consistent that every space whose dual does not contain $c_0$ or any $\ell^p$ also does not contain a barrelled subspace of dimension $\bdd$. 

If this were true for all Banach spaces, it would mean $\bdd < \barr$. (Recall that every normed vector space is contained in a Banach space.) 
At present we do not know how to prove it for all Banach spaces, only those whose duals do not contain $c_0$ or any $\ell^p$. 
But let us point out that this is a very broad class of Banach spaces. 
It was conjectured from the early days of Banach space theory (see \cite{Lind}) that every Banach space contains either $c_0$ or $\ell^p$ for some $p \geq 1$. This conjecture stood for decades, until finally refuted by Tsirelson in \cite{Tsirelson}. (For a more extreme counterexample, see \cite{Gowers}.) 
At any rate, all ``classical'' Banach spaces, i.e., all the spaces of continuous functions, or differentiable or integrable functions and their duals, all the Banach spaces used in functional analysis for the first forty years of Banach space theory, contain $c_0$ or $\ell^p$ for some $p \geq 1$. 
There is also a technical sense in which every ``simply defined'' Banach space contains $c_0$ or $\ell^p$ for some $p \geq 1$ (see \cite{Khakani}). 

Because of this, our results show that one cannot prove $\bdd = \barr$ without using more exotic, non-classical Banach spaces. 
It also means that one cannot prove that every infinite-dimensional Banach space has a barrelled subspace with dimension $\bdd$. In other words, the results from Section~\ref{sec:l1} cannot be improved by replacing $\mathrm{cov}(\M)$ with $\bdd$. 
Thus, while we have not proved the consistency of $\bdd < \barr$, we have ruled out some natural strategies for trying to show $\bdd = \barr$. 

This second main theorem is also related to Boolean algebras with the Nikodym property (defined in Section 4 below). 
The connection is forged via an observation of Schachermayer in \cite{Sch}: a Boolean algebra $\AA$ has the Nikodym property if and only if the 
subspace $Z = \mathrm{Span}\set{\chi_{[a]}}{a \in \AA}$ of $C(\mathrm{St}(\AA))$, spanned by the characteristic functions 
$\chi_{[a]}$ of all clopen subsets of the Stone space of $\AA$, is barrelled 
(where $C(\mathrm{St}(\AA))$ is the Banach space of continuous functions $\mathrm{St}(\AA) \to \R$). 
%of all continuous real-valued functions on the Stone space $St(\AA)$. 
In \cite{Sobota}, Sobota proves it is consistent for a Boolean algebra of size $<\!\continuum$ to have the Nikodym property. Later, in \cite{SZ1,SZ2,SZ3}, Sobota and Zdomskyy embark on a thorough study of just how small a Nikodym Boolean algebra can be, and prove several consistency results along these lines. The work in \cite{Sobota} includes three lower bounds: a Nikodym Boolean algebra cannot be smaller than $\bdd$, $\mathfrak{s}$, or $\mathrm{cov}(\M)$. Our theorem gives another: 
\begin{itemize}
\item Every Boolean algebra with the Nikodym property has size $\geq\! \mathrm{cov}(\N)$. 
\end{itemize}
\noindent This is Theorem~\ref{thm:Nik} below. It follows from our second main theorem because of Schachermayer's observation, and because for any compact Hausdorff space $K$, the dual of the Banach space $C(K)$ contains an isometrically embedded copy of $c_0$. 

In what follows, a ``Banach space'' may be taken as a (complete normed) vector space over either $\R$ or $\CC$. We assume the reader is familiar with the basic theory of such spaces, and recommend \cite{Swartz} as a reference. We use some topology and descriptive set theory in what follows, our basic references for which are \cite{Engelking} and \cite{Kechris}, but we assume less familiarity with these topics. Although we do prove some consistency results, no knowledge of forcing  is needed, with just one exception, the proof of Corollary~\ref{cor:l1+}, and this can be skipped by those who wish to without losing the main thread of the paper.

%%%%%%%%%%%%
\section{The separable case of the first main theorem}\label{sec:l1}
%%%%%%%%%%%%

In this section we present a proof of a special case of the first main theorem mentioned in the introduction. The special case states that every separable Banach space has a barrelled subspace of cardinality $\mathrm{non}(\M)$. 
We have chosen to present this case separately because the proof is simpler in this case, involving fewer preliminaries from topology and set theory, and some readers may be most interested only in this special case. 
Furthermore, the general theorem can be proved simply by ``adding on'' some set-theoretic generalizations and observations, without the need to repeat the main part of the argument from the separable case. 
Thus it is our hope that dividing things in this way adds to the clarity and readability of the paper, without significantly adding to its length.

\vspace{1.5mm}

Barrelled spaces are locally convex topological vector spaces for which every barrel 
(absolutely convex, closed, absorbing set) is a neighborhood of zero. These 
spaces are important because such fundamental theorems  
as the Uniform Boundedness Principle and the Closed Graph and Open Mapping 
Theorems hold in them. See \cite{Swartz} for more information on barrelled spaces. 
In what follows, we use the following fact about barrelled spaces; the reader otherwise unfamiliar with barrelled spaces may take this as the definition of \emph{barrelled} for the purposes of this paper:
\begin{itemize}
\item[$\circ$] Let $X$ be a Banach space and $X'$ its dual space. A subspace $Z$ of $X$ is barrelled if and only if for any norm-unbounded sequence $\seq{\by_m}{m \in \w}$ in $X'$, i.e. a sequence with $\sup_{m \in \w}\norm{\by_m} = \infty$, there is some $\bx \in Z$ such that $\sup_{m \in \w}|\by_m(\bx)| = \infty$. (See, e.g., \cite[Theorem 5.21 and Corollary 24.5]{Swartz}.)
\end{itemize}

Recall that the \emph{density character} of a Banach space $X$ is the cardinal number 
$$d(X) \,=\, \min\set{|D|}{D \sub X \text{ and } \closure{D} = X},$$
i.e., it is the least size of a topologically dense subset of $X$. The topological literature typically refers to this as the \emph{density} of $X$ (for any space $X$, not just Banach spaces), but the term ``density character'' is more common for Banach spaces. 

%As stated in the introduction, the goal of this section is to generalize the results of Section~\ref{sec:l1} from $\ell^1$ to arbitrary Banach spaces $X$. To do this, we need a set to play the part for $X$ that $\set{\be_i}{i \in \w}$ played for $\ell^1$. This part will be played by norming sets. 
Recall that a set $E$ of unit vectors in a Banach space $X$ is called \emph{norming} if for any $\by \in X'$, the dual of $X$, $\norm{\by} = \sup_{\be \in E}|\by(\be)|$. 

\begin{observation}\label{obs:norming}
An infinite-dimensional Banach space has density character $\k$ if and only if it has a norming set of size $\k$. 
\end{observation}
\begin{proof}
On the one hand, if $D \sub X$ is a dense subset of $X$ then $\set{\frac{1}{\norm{\bd}}\bd}{\bd \in D \setminus \{0\}}$ is a norming set. Thus $X$ has a norming set of cardinality $d(X)$. 

On the other hand, suppose $N$ is a norming set for $X$. 
While a norming set may not be dense in the unit sphere, the span of a norming set is dense in $X$. (This is a consequence of the Hahn-Banach Theorem.) 
It follows that 
$$D \,=\, \set{c_1\be_1 + \dots + c_n\be_k}{c_1,\dots,c_k \in \Q \text{ and } \be_1,\dots,\be_k \in N}$$ 
is dense in $X$. And if $N$ is infinite (which must be the case if $X$ is infinite-dimensional), then $|D| = |N|$, and thus $d(X) \leq |N|$. 
\end{proof}

We adopt the convention that a sequence in a set $E$ is a function $\w \to E$, where $\w$ denotes the first countable ordinal $\w = \{0,1,2,\dots\} = \N \cup \{0\}$. 
Adopting the usual notation for ordinals, we take $n = \{0,1,2,\dots,n-1\}$ for all $n \in \w$. This is merely a convenience, which enables us to write $f \rest n$ instead of $f \rest \{0,1,2,\dots,n-1\}$. 
If $s$ is a function with $\mathrm{dom}(s) = n = \{0,1,2,\dots,n-1\}$, then we write $|s| = n$. 
 
The space $\w^\w$ of all functions $\w \to \w$ is called the \emph{Baire space}. 
The standard basis for the Baire space consists of sets of the form
$$\tr{s} \,=\, \set{x \in \w^\w}{x \rest |s| = s},$$
where $s$ is a function $n \to \w$. 
%If $s: n \to \w$ and $t: m \to \w$, then we define the concatenation $s \cat t : (n+m) \to \w$ by setting
%$$(s \cat t)(i) \,=\, \begin{cases}
%s(i) &\text{ if } i < n, \\
%t(i-n) &\text{ if } n \leq i < n+m.
%\end{cases}$$ 
Observe that $\tr{t} \sub \tr{s}$ if and only if $t$ is an extension of $s$. 
The Baire space is separable and completely metrizable, making it an example of a Polish space. 
It is homeomorphic to the space $\R \setminus \Q$ of irrational numbers. 

Let $\I$ denote the set of injective functions $\w \to \w$. This is a subspace of $\w^\w$, and in fact it is a $G_\dlt$ subspace. To see this, observe that 
$$U_n \,=\, \set{x \in \w^\w}{x \rest n \text{ is injective}\vphantom{f^2}} \,=\, \textstyle \bigcup \set{\tr{s}}{|s|=n \text{ and $s$ is injective}\vphantom{f^2}}$$
is an open set for every $n \in \w$, and this means that $\I = \bigcap_{n \in \w}U_n $ is $G_\dlt$. 
Recalling that a subspace of a completely metrizable space is itself completely metrizable if and only if it is $G_\dlt$ \cite[Theorems 4.3.23-24]{Engelking}, this shows $\I$ is completely metrizable, hence Polish. In fact, by the Alexandrov-Urysohn characterization of the Baire space, $\I$ is homeomorphic to $\w^\w$ \cite[Exercise 7.2G]{Kechris}. In context, we write just $\tr{s}$, rather than $\tr{s} \cap \I$, to denote the basic open neighborhood of $\I$ described by an injective function $s: n \to \w$. 

Let $X$ be a separable Banach space. By Observation~\ref{obs:norming}, $X$ has a countable norming set. Let $N = \set{\be_i}{i \in \w}$ be (an enumeration of) a norming set in $X$. Given $f \in \I$, define 
$$\bx_f^N \,=\, \textstyle \sum_{n \in \w}\frac{1}{3^n}\be_{f(n)}.$$ 
Because $N$ contains only unit vectors, this sum converges and $\bx_f^N \in X$. 
Thus the map $f \mapsto \bx_f^N$ gives us a natural way to identify $\I$ with a subset of $X$. 

\begin{lemma}\label{lem:MainLemma} 
Let $X$ be a separable Banach space, and let $N = \set{\be_i}{i \in \w}$ be a norming set for $X$. 
Suppose $\seq{\by_m}{m \in \w}$ is a sequence in the dual space $X'$ with $\sup_{m \in \w}\norm{\by_m} = \infty$, but $\sup_{m \in \w} \big| \by_m( \be_i )\big| < \infty$ for all $i \in \w$. 
Then 
$$\textstyle \set{f \in \I}{\sup_{m \in \w} \big| \by_m(\bx_f) \big| = \infty}$$ 
is a co-meager subset of $\I$.
\end{lemma}
\begin{proof}
For each $i \in \w$, define  
$B_i = \sup_{m \in \w} \big| \by_m( \be_i )\big|$, which is finite by hypothesis. 
Given an injective function $s: n \to \w$ and given some $k \leq n$, define 
$$M_k^s \,=\,\textstyle \max \Big\{k,\frac{3}{2}B_{s(0)},\frac{3}{2}B_{s(1)},\dots,\frac{3}{2}B_{s(k-1)} \Big\}.$$
Observe that the value of $M^s_k$ does not depend on $s(i)$ for any $i \geq k$. In particular, if $f: \w \to \w$ is injective and $k \leq m < n$, then $M^{f \restriction m}_k = M^{f \restriction n}_k$. In light of this, let $M^f_k$ denote the common value of $M^{f \restriction n}_k$ for all $n > k$. 

Let us say that $f: \w \to \k$ satisfies property $(*)$ at some $k \in \w$ if 
$$(*): \quad \text{There is some $m \in \w$ such that $\big| \by_m(\be_{f(k)}) \big| \,>\, \norm{\by_m}-1 \,>\, k3^kM^f_k$.}$$
In this case, we say that $m \in \w$ \emph{witnesses} property $(*)$ at $k$ for $f$. 
Roughly, $(*)$ says that there is some $m \in \w$ such that $\big| \by_m(\be_{f(k)}) \big|$ is very large compared to $M^f_k$, and is also close to the norm of $\by_m$. 

For each $n \in \w$, let 
\begin{align*}
U_n &\,=\, \set{f \in \I}{f \text{ satisfies }(*) \text{ at $k$ for at least $n$ different values of $k$}} \\
&\,=\, \textstyle \bigcup \set{\tr{s}}{s \text{ satisfies }(*) \text{ at $k$ for at least $n$ different values of $k$}}.
\end{align*}
This set is clearly open, and we claim it is also dense in $\I$. 

Fix $n \in \w$. 
To see that $U_n$ is dense in $\I$, fix an injective function $s: k \to \w$ for some $k \in \w$, so that $\tr{s}$ is a basic open subset $\I$. 
Using our hypothesis that $\sup_{m \in \w}\norm{\by_m} = \infty$, first pick some $m \in \w$ such that 
$\norm{\by_m} > k3^kM_k^s+1$. 
Then, using the fact that $\set{\be_i}{i \in \w}$ is a norming set, 
pick some $j \in \w$ such that $|\by_m(\be_j)| > \norm{\by_m}-1$. Define $s': (k+1) \to \k$ by setting
$$s'(i) \,=\, \begin{cases}
s(i) &\text{ if } i < k \\
j &\text{ if }i = k.
\end{cases}$$
Observe that $|\by_m(\be_j)| > k3^kM^s_k > B_{s(i)} \geq |\by_m(\be_{s(i)})|$ for all $i < k$. It follows that $j \neq s(i)$ for any $i < k$, and therefore $s'$ is injective. 
Furthermore, our choices of $m$ and $j$ ensure that $s'$ satisfies property $(*)$ at $k = |s|$. 

Now fix an injective function $s_0: k \to \w$ for some $k \in \w$. 
Applying the argument from the previous paragraph $n$ times in a row, we can obtain sequences $s_1,s_2,\dots,s_n$ such that 
for each $i < n$, $s_{i+1} \rest |s_i| = s_i$ and $s_{i+1}$ satisfies $(*)$ at $|s_i|$. 
But then $s_n$ satisfies $(*)$ for at least $n$ different values of $k$, namely for $k = |s_0|,|s_1|,\dots,|s_{n-1}|$. Thus $\tr{s_n} \sub U_n$. 
Also, because $s_{i+1} \rest |s_i| = s_i$ for all $i < n$, 
$$\tr{s_0} \supseteq \tr{s_1} \supseteq \tr{s_2} \supseteq \dots \supseteq \tr{s_n}.$$ 
Therefore $\tr{s_n} \sub \tr{s_0} \cap U_n$. As $\tr{s_0}$ was an arbitrary basic open subset of $\I$, this shows $U_n$ is dense in $\I$.

Because each of the $U_n$ is an open dense subset of $\I$, 
$$X \,=\, \textstyle \bigcap_{n \in \w}U_n \,=\, \set{f \in \I}{f \text{ satisfies }(*) \text{ at $k$ for infinitely many values of $k$}}$$ 
is a dense $G_\dlt$ subset of $\I$. To finish the proof of the lemma, it suffices to show that if $f \in X$ then $\sup_{m \in \w} \big| \by_m(\bx_f^N) \big| = \infty$. 

Fix $f \in X$, fix some $k \in \w$ such that $f$ satisfies $(*)$ at $k$, and fix some $m \in \w$ such that $m$ witnesses property $(*)$ at $k$ for $f$. 
Because $\by_m$ is linear,  
\begin{align*}
M^f_k &\,=\,\textstyle \max \Big\{k,\frac{3}{2}B_{f(0)},\frac{3}{2}B_{f(1)},\dots,\frac{3}{2}B_{f(k-1)} \Big\} \\
&\,\geq\, \textstyle \sup \set{ \,\big| \by_m\big( \sum_{i < k} x_i\be_{f(i)} \big)\big| }{ \sum_{i < k} |x_i|  \leq \frac{3}{2}} \\
&\,\geq\,\textstyle \big| \by_m \big( \sum_{i < k}\frac{1}{3^i}\be_{f(i)} \big) \big|.
\end{align*}

Using this fact together with the triangle inequality, the inequalities in $(*)$, and the fact that $M^f_k \geq k$, we get the following lower bound on $\big| \by_m(\bx_f^N) \big|$:
\begin{align*}
\big| \by_m(\bx_f^N) \big| &\,=\,\textstyle \big|\, \by_m \big( \sum_{i \in \w}\bx_f^N(f(i))\be_{f(i)} \big)\,\big| \\ 
&\,=\,\textstyle \big|\, \by_m \big( \sum_{i \in \w}\frac{1}{3^i}\be_{f(i)} \big)\,\big| \\ 
&\,=\,\textstyle \big|\, \by_m \big( \sum_{i < k}\frac{1}{3^i}\be_{f(i)} \big) + \by_m \big( \frac{1}{3^k}\be_{f(k)} \big) + \by_m \big( \sum_{i > k}\frac{1}{3^i}\be_{f(i)} \big) \,\big| \\ 
&\,\geq\, \textstyle \big| \by_m(\frac{1}{3^k}\be_{f(k)}) \big| - \big| \by_m \big( \sum_{i < k}\frac{1}{3^i}\be_{f(i)} \big) \big| - \big| \by_m \big( \sum_{i > k}\frac{1}{3^i}\be_{f(i)} \big) \big| \\ 
&\,\geq\,\textstyle \frac{1}{3^k}\big| \by_m(\be_{f(k)}) \big| - M^f_k - \norm{\by_m}_\infty \bignorm{ \sum_{i > k}\frac{1}{3^i}\be_{f(i)} }_1 \\ 
&\,\geq\,\textstyle \frac{1}{3^k}\big| \by_m(\be_{f(k)}) \big| - M^f_k - \frac{1}{2 \cdot 3^k}\norm{\by_m}_\infty \\ 
&\,\geq\,\textstyle \frac{1}{3^k} \Big( 3^k \big| \by_m(\be_{f(k)}) \big| - 3^kM^f_k - \frac{1}{2}\norm{\by_m}_\infty \Big) \\ 
&\,\geq\,\textstyle \frac{1}{3^k} \Big( 3^k (\norm{\by_m}_\infty-1) - \frac{1}{k}(\norm{\by_m}_\infty-1) - \frac{1}{2}\norm{\by_m}_\infty \Big) \\ 
&\,\geq\,\textstyle \frac{1}{3^k} \Big( (3^k-1)\norm{\by_m}_\infty - 3^k \Big) \\ 
&\,\geq\,\textstyle \big( 1-\frac{1}{3^k} \big) M^f_k-1 \\ 
&\,\geq\,\textstyle \frac{2}{3} k - 1.
\end{align*}
Because $f \in X$, this shows that there are arbitrarily large values of $k$ such that $\big| \by_m(\bx_f^N) \big| \geq \frac{2}{3}k - 1$ for some $m \in \w$. 
Hence $\sup_{m \in \w} \big| \by_m(\bx_f^N) \big| = \infty$.
\end{proof}

\begin{theorem}\label{thm:nonmeager}
Let $X$ be a separable Banach space, and let $N = \set{\be_i}{i \in \w}$ be a norming set for $X$. 
If $Y$ is a non-meager subset of $\I$, then 
$$Z \,=\, \mathrm{Span}\big(\! \set{\be_i}{i \in \w} \cup \big\{ \bx_f^N :\, f \in Y \big\} \big)$$
is a barrelled subspace of $X$.
\end{theorem}
\begin{proof}
Fix a sequence $\seq{ \by_m }{m \in \w}$ in $X'$ such that $\sup_{m \in \w} \norm{\by_m} = \infty$.  %(a norm-unbounded sequence). 
We wish to show that $\sup_{m \in \w} |\by_m(\bx)| = \infty$ for some $\bx \in Z$. 
If $\sup_{m \in \w} | \by_m(\be_i) | = \infty$ for some $i \in \w$, then we are done, because $\be_i \in Z$. 
If not, then by the previous lemma, $\textstyle \set{f \in \I}{\sup_{m \in \w} \big| \by_m(\bx_f^N) \big| = \infty}$ 
is a co-meager subset of $\I$. 
But $Y$ meets every co-meager subset of $\I$, so there is some $f \in Y$ with $\sup_{m \in \w} \big| \by_m(\bx_f^N) \big| = \infty$. %In other words, $\seq{\by_m}{m \in \w}$ is not pointwise bounded on $X$ (or $Z$). 
\end{proof}

\begin{theorem}\label{thm:l1}
Every separable Banach space has a barrelled subspace with dimension $\leq\! \mathrm{non}(\M)$. 
In particular, $\barr \leq \mathrm{non}(\M)$.
\end{theorem}
\begin{proof}
Let $X$ be a separable Banach space and let $N = \set{\be_i}{i \in \w}$ be a norming set for $X$. 
Because $\I$ is homeomorphic to the Baire space $\w^\w$, which in turn is homeomorphic to $\R \setminus \Q$, there is a non-meager subset of $\I$ with size $\mathrm{non}(\M)$. 
If $Y$ is such a set, then $Z = \mathrm{Span}\big(\! \set{\be_i}{i \in \w} \cup \big\{ \bx_f^N :\, f \in Y \big\} \big)$ barrelled, by the previous theorem, and $\mathrm{dim}(Z) \leq |Y \cup \set{\be_i}{i \in \w}| = \mathrm{non}(\M)$. 
From this, and the definition of $\barr$, it follows immediately that $\barr \leq \mathrm{non}(\M)$. 
\end{proof}

We end this section with an observation about rearrangements of the series $\sum_{n \in \w}\nicefrac{1}{3^n}$. 
Let $\ell^1$ is the Banach space of all absolutely summable sequences (in $\R$ or $\CC$ -- it does not matter for the following argument) and 
or each $i \in \w$, let $\be_i \in \ell^1$ denote the sequence that is equal to $1$ in coordinate $i$ and $0$ elsewhere. 
Note that $N = \set{\be_i}{i \in \w}$ is a norming set for $\ell^1$. 

In this context, we say that $\bx_f^N \in \ell^1$ is a \emph{rearrangement} of $\seq{\frac{1}{3^n}}{n \in \w}$ if $f$ is not merely injective, but bijective. 

\begin{observation}
$\textstyle \set{f \in \I}{f \text{ is a bijection}}$ 
is a dense $G_\dlt$ subset of $\I$.
\end{observation}
\begin{proof}
For each $n \in \w$, observe that  
$$U_n \,=\, \set{f \in \I}{n \in \mathrm{Image}(f)} \,=\, \textstyle \bigcup \set{\tr{s}}{s \text{ is injective and } n \in \mathrm{Image}(s)}$$ 
is a nonempty open subset of $\I$. 
Given any nonempty open $V \sub \I$, fix some basic open $\tr{s} \sub V$. If $n \in \mathrm{Image}(s)$ then $\tr{s} \sub U_n \cap V$. If not, define $t: (|s|+1) \to \w$ by setting $t \rest |s| = s$ and $t(|s|) = n$; then $t$ is injective and $\tr{t} \sub U_n \cap V$. Either way, $U_n \cap V \neq \0$. As $V$ was arbitrary, this shows $U_n$ is dense in $\I$. 
Consequently, $\textstyle \set{f \in \I}{f \text{ is a bijection}} = \bigcap_{n \in \w}U_n$ is a dense $G_\dlt$ set. 
\end{proof}

It follows that in the statement of Theorem~\ref{thm:nonmeager} with $X = \ell^1$ and $N$ as described above, we may, without loss of generality, assume every $f \in X$ is bijective, so that each $\bx_f$ is a rearrangement. In other words, combining the above observation with the preceding proofs shows: 
\begin{itemize}
\item[$\circ$] There is a set $Y$ of rearrangements of $\seq{\frac{1}{3^n}}{n \in \w}$ such that $|Y| = \mathrm{non}(\M)$ and $Z = \mathrm{Span}\big(\! \set{\be_i}{i \in \w} \cup Y \big)$ is barrelled.
\end{itemize}
With a little more work, one may also show that if $\sup_{m \in \w}|\by_m(\be_i)| = \infty$, then 
$\set{f \in \I}{\sup_{m \in \w}|\by_m(\bx_f^N)| = \infty}$ is co-meager in $\I$. We omit the details of this, but point out that this observation can be used to improve the above to 
\begin{itemize}
\item[$\circ$] There is a set $Y$ of rearrangements of $\seq{\frac{1}{3^n}}{n \in \w}$ such that $|Y| = \mathrm{non}(\M)$ and $Z = \mathrm{Span}(Y)$ is barrelled.
\end{itemize}
In other words, the $X = \ell^1$ case of Theorem~\ref{thm:l1} is still true for the span of some $\mathrm{non}(\M)$ rearrangements of $\seq{\frac{1}{3^n}}{n \in \w}$.

Ruckle proved in \cite{Ruckle} that for any $\bx \in \ell^1$, the span of all rearrangements of $\bx$ is a barrelled space. This result was also derived as a consequence of a more general theorem by the second author in \cite{Stuart}. The above observation is in this vein, though it neither directly implies the Ruckle result (because it is unclear whether our theorem holds for arbitrary $\bx \in \ell^1$ in the place of $\seq{\frac{1}{3^n}}{n \in \w}$), nor is implied by it (because we need only $\mathrm{non}(\M)$ rearrangements to get a barrelled span).

%%%%%%%%%%%%
\section{Generalization to all Banach spaces}\label{sec:generalization}
%%%%%%%%%%%%

The goal of this section is to generalize the results of the last section to Banach spaces with arbitrary density character. The main change is that instead of dealing with the familiar Baire space $\w^\w$, we must deal with its less familiar generalization $D^\w$, where $D$ is a discrete space of uncountable size. 

Following the standard set-theoretic conventions, a cardinal $\k$ also denotes a set of size $\k$ (the set of all ordinal numbers $<\!\k$). 
Given an infinite cardinal $\k$, $\k^\w$ denotes the space of functions $\w \to \k$. 
The topology on $\k^\w$ is obtained by giving $\k$ the discrete topology and then putting the usual product topology on $\k^\w$. 
This makes $\k^\w$ a completely metrizable space of weight $\k$, sometimes known as the \emph{generalized Baire space} of weight $\k$ (e.g. in \cite{Engelking}). 
Just as with the usual Baire space $\w^\w$, the standard basis for this space consists of sets of the form 
$$\tr{s} \,=\, \set{f \in \k^\w}{ f \rest |s| = s},$$
where $s$ is a function $n \to \k$ for some $n \in \w$. 

Generalizing the space $\I$ from the previous section, let
$$\I_\k = \set{f \in \k^\w}{f \text{ is injective}}.$$ 
As before, $\I_\k$ is a $G_\dlt$ subset of $\k^\w$. To see this, observe that
$$U_n \,=\, \set{x \in \k^\w}{x \rest n \text{ is injective}} \,=\, \textstyle \bigcup \set{\tr{s}}{|s|=n \text{ and } s \text{ is injective}}$$
is open for every $n \in \w$, and $\I_\k = \bigcap_{n \in \w}U_n$. 
This implies $\I_\k$ is completely metrizable. 
(In fact, although this is not really needed here, one can show via \cite[Theorem 7.3.8]{Engelking} and \cite{Stone} that $\I_\k$ is homeomorphic to $\k^\w$ itself. This is analogous to the fact stated in the previous section that $\I$ is homeomorphic to $\w^\w$.) 
For convenience, if $s$ is an injective function $n \to \k$ then we write $\tr{s}$ rather than $\tr{s} \cap \I_\k$ to denote the basic open neighborhood of $\I_\k$. 

Let $X$ be a Banach space with density character $\leq\! \k$, and let $N = \set{\be_i}{i \in \k}$ be (an enumeration of) a norming set in $X$. Given $f \in \I_\k$, define 
$$\bx_f^N \,=\, \textstyle \sum_{n \in \w}\frac{1}{3^n}\be_{f(n)}.$$ 
Because $N$ contains only unit vectors, this sum converges and $\bx_f^N \in X$. 
Thus the map $f \mapsto \bx_f^N$ gives us a natural way to identify $\I_\k$ with a subset of $X$. 

The following lemma is analogous to Lemma~\ref{lem:MainLemma} in the previous section.

\begin{lemma}
Let $X$ be a Banach space, and let $N = \set{\be_i}{i \in \k}$ be a norming set in $X$. 
Suppose $\seq{\by_m}{m \in \w}$ is a sequence in $X'$ with $\sup_{m \in \w}\norm{\by_m} = \infty$, but $\sup_{m \in \w}|\by_m(\be_i)| < \infty$ for all $i \in \k$. Then
$$\textstyle \set{f \in \I_\k}{\sup_{m \in \w} \big| \by_m(\bx_f^N) \big| = \infty}$$
is a co-meager subset of $\I_\k$.
\end{lemma}
\begin{proof}
The proof of this lemma is essentially identical to the proof of Lemma~\ref{lem:MainLemma}. 
Only three minor changes need to be made. First, every $\I$ in that proof should be replaced with an $\I_\k$. 
Second, in the two paragraphs following the definition of $U_n$, we have functions $s: k \to \w$ and $s_0: k \to \w$; but for the proof of the present lemma these should be functions $k \to \k$ instead. 
Finally, in the paragraph following the definition of $U_n$, we should choose $j \in \k$ rather than $j \in \w$. 
\end{proof}

\begin{theorem} 
Let $X$ be a Banach space, and let $N = \set{\be_i}{i \in \k}$ be a norming set for $X$. 
If $Y$ is a non-meager subset of $\I_\k$, then 
$$Z \,=\, \mathrm{Span}\big(\! \set{\be_i}{i \in \k} \cup \big\{ \bx_f^N :\, f \in Y \big\}\big)$$ 
is a barrelled subspace of $X$.
\end{theorem}
\begin{proof}
The proof of this theorem is identical to the proof of Theorem~\ref{thm:nonmeager}.
\end{proof} 

The proof of Theorem~\ref{thm:l1} generalizes to the present context less readily than those of Lemma~\ref{lem:MainLemma} and Theorem~\ref{thm:nonmeager}. The problem is that the proof of Theorem~\ref{thm:l1} relies on the fact that there is a meager subset of $\I$ with size $\mathrm{non}(\M)$. 
To generalize this, we need to know the minimum cardinality of a meager subset of $\I_\k$. 
This cardinal is known by recent work of the first author in \cite{Bideals}, but some further set theoretic preliminaries are needed to describe it. 

If $\k$ is an infinite cardinal, then $[\k]^\w$ denotes the set of all countably infinite subsets of $\k$. A family $\F \sub [\k]^\w$ is \emph{cofinal} if for every $A \in [\k]^\w$ there is some $B \in \F$ such that $B \supseteq A$. In other words, $\F \sub [\k]^\w$ is cofinal if it is (in the usual sense of the word) cofinal in the poset $([\k]^\w,\sub)$. Define
$$\cf[\k]^\w \,=\, \min \set{|\F|}{\F \text{ is a cofinal subset of }[\k]^\w}.$$ 
For example, $\cf[\w_1]^\w = \aleph_1$ because, identifying an ordinal with the set of its predecessors as usual, $\set{\a}{\w \leq \a < \w_1}$ is a cofinal subset of $[\w_1]^\w$. In fact, this example with $\w_1$ can be extended to show that $\cf[\w_n]^\w = \aleph_n$ for all $n \in \w$. 

It is not difficult to see that $\k \leq \cf[\k]^\w$ for any uncountable cardinal $\k$. Furthermore, a reasonably straightforward diagonalization argument can be used to show that $\k^+ \leq \cf[\k]^\w$ whenever $\k$ is a singular cardinal with cofinality $\w$. \emph{Shelah's Strong Hypothesis}, abbreviated \ssh, is the statement that these straightforward lower bounds are sharp:
\begin{itemize}
\item[$\ssh:$] $\qquad \qquad \qquad \quad  \cf[\k]^\w \,=\, \begin{cases}
\k &\text{ if } \cf(\k) > \w, \\
\k^+ &\text{ if } \cf(\k) = \w.
\end{cases}$
\end{itemize}
Let us note that \ssh is consistent with \zfc. In fact this is true in a rather strong sense: the failure of \ssh implies the consistency of certain large cardinal axioms, and therefore the consistency of \zfc$+\,\neg$\ssh is only provable in strong extensions of \zfc (see \cite{JMPS} or \cite{Bspec}).

\begin{theorem}\label{thm:l1+}
If $X$ is a Banach space with density character $\k$, then $X$ has a barrelled subspace with dimension at most $\mathrm{cf}[\k]^\w \cdot \mathrm{non}(\M)$. 
\end{theorem}
\begin{proof} 
The space $\I_\k$ is completely metrizable, and it has a basis of size $\k$: in fact the standard basis, consisting of the sets of the form $\tr{s}$, has size $\k$. 
By \cite[Theorem 3.4]{Bideals}, this implies $\I_\k$ has a non-meager subset $Y$ of size $\mathrm{cf}[\k]^\w \cdot \mathrm{non}(\M)$. 
But then $Z = \mathrm{Span}\big(\! \set{\be_i}{i \in \k} \cup Y \big)$ is barrelled by the previous theorem, 
and $\mathrm{dim}(Z) \leq \big|\! \set{\be_i}{i \in \k} \cup Y \big| = \k \cdot \cf[\k]^\w \cdot \mathrm{non}(\M) = \mathrm{cf}[\k]^\w \cdot \mathrm{non}(\M)$.
\end{proof} 

Observe that $\cf[\w]^\w = 1$, because the family $\{\w\}$ is cofinal in $\cf[\w]^\w$, and thus setting $\k = \w$ in Theorem~\ref{thm:l1+} immediately yields the special case proved in the previous section.  

\begin{corollary}\label{cor:l1+}
For every cardinal number $\lambda$, there is a cardinal-preserving notion of forcing which forces $\continuum \geq \lambda$ and also forces that every Banach space with density character $<\! \continuum$ has a barrelled subspace with dimension $<\! \continuum$. 

In other words, it is consistent with arbitrarily large values of $\continuum$ that for every infinite $\k < \continuum$, every Banach space with density character $<\! \continuum$ has a barrelled subspace with dimension $<\! \continuum$. 
\end{corollary}
\begin{proof}
Fix $\lambda$, and fix some $\mu \geq \lambda$ such that $\cf(\mu) > \w$, and such that $\cf[\k]^\w < \mu$ whenever $\k < \mu$. (Assuming \ssh, this is true for some $\mu \in \{\lambda,\lambda^+,\lambda^{++}\}$. In \zfc, it is true for any $\beth$-fixed point above $\lambda$ with uncountable cofinality.) Then let $\mathbb P$ be the usual notion of forcing to add $\mu$ Cohen reals. 
Because $\mathbb P$ has the ccc, it does not change the value of $\cf[\k]^\w$ for any cardinal $\k$. Furthermore, $\mathbb P$ forces $\mathrm{non}(\M) = \aleph_1$. By our choice of $\mu$, it follows that $\mathrm{cf}[\k]^\w \cdot \mathrm{non}(\M) = \mathrm{cf}[\k]^\w < \mu$ for any uncountable $\k < \mu$ in the forcing extension. Thus the corollary follows from Theorem~\ref{thm:l1+}.
\end{proof}

%%%%%%%%%%%%
\section{How small can a barrelled subspace be?}\label{sec:random}
%%%%%%%%%%%%

We say that a Banach space $X$ \emph{contains} a Banach space $Y$ if there is an injective bounded linear map $\phi: Y \to X$, or equivalently, if $Y$ is linearly homeomorphic to a (closed) subspace of $X$. 
Recall that $\phi$ is \emph{bounded} if there is some $b > 0$ such that $\norm{\phi(\bx)} \leq b \norm{\bx}$ for all $\bx \in X$, and that when this is the case, then by the Open Mapping Theorem, there is some $b' > 0$ such that $\frac{1}{b} \norm{\bx} \leq \norm{\phi(\bx)}$ for all $\bx \in X$. 
Thus a linear map $\phi$ is bounded if and only if there is some $b > 0$ such that $\frac{1}{b} \norm{\bx} \leq \norm{\phi(\bx)} \leq b \norm{\bx}$ for all $\bx \in X$.

In this section we prove that any Banach space whose dual contains either $c_0$ or $\ell^p$ for some $p \geq 1$ has no barrelled subspace with dimension $<\!\mathrm{cov}(\NN)$. 

Let $X$ be a Banach space, and let $X'$ denote its dual. 
A sequence $\seq{\by_m}{m \in \w}$ in $X'$ is \emph{norm-unbounded} if $\sup_{m \in \w}\norm{\by_m} = \infty$. 
Let us say that $\bx \in X$ \emph{detects the norm-unboundedness} of a sequence $\seq{\by_m}{m \in \w}$ if $\sup_{m \in \w} \big| \by_m(\bx) \big| = \infty$. 
Let us say that $Y \sub X$ \emph{detects the norm-unboundedness of all functional sequences} if for any norm-unbounded sequence $\seq{\by_m}{m \in \w}$ in $X'$, there is some $\bx \in Y$ that detects the norm-unboundedness of $\seq{\by_m}{m \in \w}$.

Recall that a subspace $Z$ of $X$ is barrelled if and only if $Z$ detects the norm-unboundedness of all functional sequences.  
%Moreover, the proof of Theorem~\ref{thm:nonmeager} shows that if $Y \sub X$ detects the norm-unboundedness of all functional sequences, then $Z = \mathrm{Span}(Y)$ is barrelled. The following lemma shows the converse is also true. 

\begin{lemma}\label{lem:basis}
Suppose $Z$ is a barrelled subspace of a Banach space $X$. 
If $Y \sub X$ and $Z = \mathrm{Span}(Y)$, then $Y$ detects the norm-unboundedness of all functional sequences. 
\end{lemma}
\begin{proof}
Fix $Y \sub X$ with $Z = \mathrm{Span}(Y)$. 
Our goal is to show that if $Y$ does not detect the norm-unboundedness of every functional sequence, then neither does $Z$.

To this end, suppose $\seq{\by_m}{m \in \w}$ is a norm-unbounded sequence in $X'$ such that no $\bx \in Y$ detects the norm-unboundedness of $\seq{\by_m}{m \in \w}$. 
Let $\bz \in Z$, and fix $\bx_1,\dots,\bx_n \in Y$ and $c_1,\dots,c_n$ such that $\bz = c_1\bx_1 + \dots + c_n\bx_n$. 
For each $i \leq n$, our assumption on $Y$ implies $\sup_{m \in \w}|\by_m(\bx_i)| < \infty$; let $B_i = c_i \sup_{m \in \w}|\by_m(\bx_i)|$. 
Because each of the functionals $\by_m$ is linear, 
$$\textstyle |\by_m(\bz)| \,=\, \big| \sum_{i \leq n}c_i\by_m(\bx_i) \big| \,\leq\, \sum_{i \leq n}c_i|\by_m(\bx_i)| \,\leq\, \sum_{i \leq n}c_iB_i.$$
Thus $\sup_{m \in \w}|\by_m(\bz)| \,\leq\, \sum_{i \leq n}c_iB_i$, which means that $\bz$ does not detect the norm-unboundedness of $\seq{\by_m}{m \in \w}$. 
\end{proof}

The main result of this section, the second main Theorem stated in the introduction, is proved differently for spaces whose dual contains $\ell^1$ and for spaces whose dual contains $c_0$ or $\ell^p$ for some $p > 1$. We handle the $\ell^1$ case first. 

\begin{lemma}\label{lem:containsl1}
Suppose $Y$ is a Banach space containing $\ell^1$. Then there is a sequence $\seq{\be_i}{i \in \w}$ of unit vectors in $Y$, and some $\dlt > 0$, such that 
$$\textstyle \bignorm{ \sum_{i=1}^n f(i)\be_i } \,\geq\, \dlt n$$
for every function $f: \{1,2,\dots,n\} \to \{-1,1\}$.
\end{lemma}
\begin{proof}
Let $\phi: \ell^1 \to Y$ be an injective bounded linear map, and fix $b > 0$ such that $\frac{1}{b} \norm{\bx} \leq \norm{\phi(\bx)} \leq b \norm{\bx}$ for all $\bx \in \ell^1$. 
For each $i \in \w$, let $\be^0_i$ be the sequence in $\ell^1$ with a $1$ in coordinate $i$ and a $0$ in every other coordinate, and let $\be_i = \frac{1}{\norm{\phi(\be_i^0)}}\phi(\be_i^0)$. 

Clearly $\seq{\be_i}{i \in \w}$ is a sequence of unit vectors in $Y$, and if $c_1,c_2,\dots,c_n$ is any choice of scalars, then 
\begin{align*}
\textstyle \bignorm{ \sum_{i=1}^n c_i\be_i } &\,=\, \textstyle \bignorm{ \sum_{i=1}^n \frac{c_i}{\norm{\phi(\be_i^0)}}\phi(\be_i^0) } \,=\, \bignorm{ \phi \big( \sum_{i=1}^n \frac{c_i}{\norm{\phi(\be_i^0)}}\be_i^0 \big) } \\
&\,\geq\, \textstyle \frac{1}{b} \bignorm{ \sum_{i=1}^n \frac{c_i}{\norm{\phi(\be_i^0)}}\be_i^0 } \,=\, \textstyle \frac{1}{b} \sum_{i=1}^n \frac{|c_i|}{\norm{\phi(\be_i^0)}} \,\geq\, \frac{1}{b^2} \sum_{i=1}^n |c_i|.
\end{align*}
In particular, setting $\dlt = \frac{1}{b^2}$, the lemma just asserts the special case in which $c_i = f(i) \in \{-1,1\}$ for all $i \leq n$. 
\end{proof}

\begin{theorem}\label{thm:Random1}
Let $X$ be a Banach space with the property that its dual space $X'$ contains $\ell^1$. 
If $Y \sub X$ detects the norm-unboundedness of all functional sequences, then $|Y| \geq \mathrm{cov}(\NN)$. 
\end{theorem}
\begin{proof}
Suppose $X$ is a Banach space such that $X'$ contains $\ell^1$. Applying Lemma~\ref{lem:containsl1}, fix a sequence $\seq{\be_i}{i \in \w}$ of unit vectors in $X'$ and $\dlt > 0$ such that 
for every function $f: \{1,2,\dots,n\} \to \{-1,1\}$, 
$\textstyle \bignorm{ \sum_{i=1}^n f(i)\be_i } \,\geq\, \dlt n$.

Our strategy is to define an injective mapping from a probability space $\Omega$ to norm-unbounded sequences in $X'$. 
This gives some meaning to the notion of a ``random'' norm-unbounded sequence in $X'$. 
We then show that a given $\bx \in X$ almost surely does not detect the unboundedness of a ``random" sequence in $X'$. 
Once established, this fact can be used to prove the desired inequality. 

Let $\Omega = \prod_{i \in \w}\{-1,1\}$, the space of all functions $\w \to \{-1,1\}$. 
We may (and do) view $\Omega$ as a probability space with the standard Lebesgue measure $\mu$, so that 
$$\mu \big(\! \set{s \in \Omega}{f(i)=0} \!\big) \,=\, \mu \big(\! \set{f \in \Omega}{s(i)=1} \!\big) \,=\, \textstyle \frac{1}{2}$$
for all $i \in \w$. To emphasize the probabilistic viewpoint, let 
$$P\big( \varphi(x) \big) \,=\, \mu \big(\! \set{f \in \Omega}{\varphi(f)} \!\big)$$
whenever $\varphi(x)$ is some formula/property defining a subset of $\Omega$. 
The relevant properties of $\mu$ are:
\begin{itemize}
\item[$\circ$] As a probability measure, $\mu$ is countably (sub)additive. In particular, 
$$P\big( \varphi_m(f) \text{ holds for some $m \in A$} \big) \,\leq\,\textstyle \sum_{m \in A}P\big( \varphi_m(f) \big)$$
whenever $A \sub \N$ and $\varphi_m(x)$ is some formula/propety defining a subset of $\Omega$ for every $m \in A$.
\item[$\circ$] The measure space $(\Omega,\mu)$ is an atomless probability Borel measure on a compact metrizable space. 
All such measure spaces are Borel-isomorphic to the Lebesgue measure on the unit interval $[0,1]$ \cite[Theorem 17.41]{Kechris}. 
In particular, it follows that if $\F$ is a family of $\mu$-null subsets of $\Omega$ and $|\F| < \mathrm{cov}(\NN)$, then $\bigcup \F \neq \Omega$.
\end{itemize}

For each $f \in \Omega$ and $m \in \w$, define 
$$\by_m^f \,=\, \textstyle \frac{1}{m^3}\sum_{i=1}^{m^4} f(i)\be_i$$
In other words, $\by^f_m$ is the result of a random walk in $X'$, with $m^4$ steps, where at step $i$ our walk goes in the direction of either $\be_i$ or $-\be_i$. 
For each $f \in \Omega$, 
$$\norm{\by_m^f} \,=\,\textstyle \frac{1}{m^3}\bignorm{ \sum_{i=1}^{m^4} f(i)\be_i } \,\geq\, \frac{1}{m^3}\dlt m^4 \,=\, \dlt m.$$
It follows that $\sup_{m \in \w}\norm{\by^f_m}  = \infty$ for every $f \in \Omega$. 
In other words, the mapping $f \mapsto \seq{\by^f_m}{m \in \w}$ gives us a way of associating the members of $\Omega$ to norm-unbounded sequences in $X'$. 

\begin{claim} 
For any given $\bx \in X$, 
$$P \big( \bx \text{ detects the norm-unboundedness of }\seq{\by^f_m}{m \in \w} \!\big) \,=\, 0.$$
\end{claim}

\noindent \emph{Proof of the claim:} 
Fix $\bx \in X$ and $m \in \N$. For convenience, let $x_i = \be_i(\bx)$ for each $i \leq m^4$, so that $\by^f_m(\bx) = \frac{1}{m^3}\sum_{i=1}^{m^4} f(i)x_i$ for each $f \in \Omega$. Because $\norm{\be_i} = 1$, $|x_i| \leq \norm{\bx}$ for each $i$, and in particular $\sum_{i=1}^{m^4} x_i^2 \leq m^4\norm{\bx}^2$. 

For each $i \leq m^4$, we may (and do) consider the value of each $f(i)x_i$ to be a random variable, taking each of the values $\pm x_i$ with probability $\nicefrac{1}{2}$. 
The variance of this random variable is $x_i^2$. 
The key fact we need from elementary probability theory is that for independent random variables (which these are), the variance is additive. 
This means the variance of the value of $\sum_{i=1}^{m^4} f(i)x_i$ (as $f$ varies) is $\sum_{i=1}^{m^4} x_i^2 \leq m^4\norm{\bx}^2$, and therefore standard deviation is $\leq\! m^2\norm{\bx}$. 
Furthermore, the mean of $\sum_{i=1}^{m^4} f(i)x_i$ (as $f$ varies) is $0$, by symmetry. 
By Chebyshev's inequality,  
$$\textstyle P \Big( \big| \sum_{i=1}^{m^4} f(i)x_i \big| \geq km^2\norm{\bx} \,\Big) \,\leq\, \nicefrac{1}{k^2}$$
for any $k \in \N$. Setting $k = m$ and scaling by $\frac{1}{m^3}$, it follows that 
$$\textstyle P \big( | \by^f_m(\bx) | \geq {\norm{\bx}} \,\big) \,=\, P \Big( \big| \frac{1}{m^3}\sum_{i=1}^{m^4} f(i)x_i \big| \geq {\norm{\bx}} \,\Big) \,\leq\, \nicefrac{1}{m^2}.$$
(We note that better bounds are possible, for example via the Azuma-Hoeffding inequality, but the coarser Chebyshev bound suffices to prove the claim.)

Let $\e > 0$. 
Because $\sum_{m \in \w} \frac{1}{m^2}$ converges, there is some $N \in \N$ such that $\sum_{m \geq N} \frac{1}{m^2} < \e$. 
Observe that if $m < N$ then 
$$|\by^f_m(\bx)| \,\leq\, \norm{\by^f_m}\norm{\bx} \,\leq\, \dlt m\norm{x} \,<\, \dlt mN\norm{\bx}.$$
Combining this with the previous paragraph, 
\begin{align*}
\textstyle P \big( \sup_{m \in \w}|\by_m^f(\bx)| = \infty \big) &\,\leq\, P \Big( \textstyle \sup_{m \in \w}|\by_m^f(\bx)| > \max\big\{ \dlt mN\norm{\bx},{\norm{\bx}} \big\} \Big) \\
&\,\leq\,\textstyle P \big( |\by_m^f(\bx)| > {\norm{\bx}} \text{ for some } m \geq N \big) \\
&\,\leq\, \textstyle \sum_{m \geq N} P \big( |\by_m^f(\bx)| \geq {\norm{\bx}} \big) \\
&\,<\, \textstyle \sum_{m \geq N} \frac{1}{m^2} \\ &\,<\, \e.
\end{align*}
As $\e$ was arbitrary, this implies $P \big( \sup_{m \in \w}|\by_m^f(\bx)| = \infty \big) = 0$.
\hfill $\dashv$

\vspace{2.5mm}

To finish the proof of the theorem, let $Y \sub X$. Observe that 
\begin{align*}
&\set{f \in \Omega}{ \text{$Y$ detects the norm-unboundedness of }\seq{y^f_m}{m \in \w} } \\
& \qquad \,=\, \textstyle \bigcup_{\bx \in Y} \set{f \in \Omega}{ \text{$\bx$ detects the norm-unboundedness of }\seq{y^f_m}{m \in \w} }.
\end{align*}
By the claim, this is a size-$|Y|$ union of $\mu$-null sets. I $|Y| < \mathrm{cov}(\NN)$ then, because $\mu$ is Borel isomorphic to the standard Lebesgue measure on $[0,1]$, 
$$\textstyle \bigcup_{\bx \in Y} \set{f \in \Omega}{ \text{$\bx$ detects the norm-unboundedness of }\seq{y^f_m}{m \in \w} } \,\neq\, \Omega.$$ 
In particular, if $|Y| < \mathrm{cov}(\NN)$ then $Y$ does not detect the norm-unboundedness of every sequence of the form $\seq{\by^f_m}{m \in \w}$, and so $Y$ does not detect the norm-unboundedness of all functional sequences.
\end{proof}

\begin{theorem}\label{thm:RandomWalks}
If $X$ is a Banach space whose dual contains $\ell^1$, then every barrelled subspace of $X$ has dimension at least $\mathrm{cov}(\NN)$. 
\end{theorem}
\begin{proof}
Suppose $Z$ is a barrelled subspace of $X$, and fix $Y \sub Z$ with $|Y| = \mathrm{dim}(Z)$ and $Z = \mathrm{Span}(Y)$. 
By Lemma~\ref{lem:basis}, $Y$ detects the norm-unboundedness of all functional sequences. 
By Theorem~\ref{thm:Random1}, this implies $\mathrm{dim}(Z) = |Y| \geq \mathrm{cov}(\NN)$. 
\end{proof}

\begin{corollary}\label{cor:Con1}
It is consistent with \zfc that no Banach space whose dual contains $\ell^1$ has a barrelled subspace with dimension $\bdd$. 
%In particular, it is consistent that none of the spaces listed in Lemma~\ref{lem:Ddagger} has a barrelled subspace with dimension $\bdd$. 
\end{corollary}
\begin{proof}
This follows from the previous corollary, together with the fact (well-known to set theorists) that the inequality $\bdd < \mathrm{cov}(\NN)$ is consistent with \zfc. For example, this inequality holds in the so-called ``random real model'' (see \cite{Blass}).
\end{proof}

Some well known Banach spaces $X$ with the property that $X'$ contains $\ell^1$ are:
\begin{itemize}
\item[$\circ$] $\ell^\infty$, or more generally, $\ell^\infty(\k)$ for any infinite cardinal $\k$, 
\item[$\circ$] $c_0$ and $c$, or more generally, $c_0(\k)$ and $c(\k)$ for any infinite cardinal $\k$, 
\item[$\circ$] $L^\infty(\Omega,\Sigma,\mu)$ for any non-atomic measure space $(\Omega,\Sigma,\mu)$, 
\item[$\circ$] the function space $C(K)$ for any infinite compact Hausdorff space $K$.
\end{itemize}

We provide a proof of this for $C(K)$, since this fact is relevant to the application of Theorem~\ref{thm:RandomWalks} to Nikodym Boolean algebras mentioned in the introduction. 

\begin{lemma}\label{lem:CK}
Let $K$ be an infinite compact Hausdorff space, and let $X = C(K)$. Then $X'$ contains an isometrically embedded copy of $\ell^1$. 
\end{lemma}
\begin{proof}
Let $K$ be an infinite compact Hausdorff space. 
Recall that a subset of $K$ is \emph{relatively discrete} if every point of $D$ is isolated in $D$, i.e., the subspace topology on $D$ is the discrete topology. 
Because $K$ is infinite and Hausdorff, there is an infinite relatively discrete $D = \set{d_i}{i \in \N} \sub K$. 
(This fact is folklore, but we include a short proof for completeness. 
First, we claim that if $H$ is any compact Hausdorff space, then there is an infinite subspace $H'$ with an isolated point. To see this, pick $x,y \in H$ with $x \neq y$ and disjoint open neighborhoods $U \ni x$ and $V \ni y$. Because $H$ is infinite, at least one of $U$ or $V$ has infinite complement, and therefore at least one of $\{x\} \cup (K \setminus U)$ or $\{y\} \cup (K \setminus V)$ is as required. Now apply this observation recursively, as follows. First find an infinite $H_1 \sub K$ with an isolated point $d_1$, then find an infinite $H_2 \sub H_1 \setminus \{d_1\}$ with an isolated point $d_2$, then an infinite $H_3 \sub H_2 \setminus \{d_2\}$ with an isolated point $d_3$, etc.) 

For each $i \in \N$, let $\be_i$ denote the functional on $C(K)$ that maps $f \in C(K)$ to $f(d_i)$: i.e., $\be_i$ is the evaluation functional at $d_i$. Then define $\phi: \ell^1 \to C(K)'$ by 
$$\textstyle \phi(\seq{a_i}{i \in \N}) \,=\, \sum_{i \in \N}a_i\be_i.$$
This map is clearly linear and, because $\sum_{i \in \N}a_i$ is absolutely convergent whenever $\seq{a_i}{i \in \N} \in \ell^1$, the right hand side converges to a functional in $C(K)'$. So we have a linear mapping $\ell^1 \to C(K)'$, and it is not difficult to see that this mapping is injective. We wish to show that $\norm{\phi(\bx)} = \norm{\bx}$ for all $\bx \in \ell^1$.  

Fix $\seq{a_i}{i \in \N} \in \ell^1$. In one direction, the triangle inequality gives
$$\textstyle \norm{\phi(\seq{a_i}{i \in \N})} \,=\, \bignorm{\sum_{i \in \N}a_i\be_i} \,\leq\, \sum_{i \in \N}\norm{a_i\be_i} \,=\, \sum_{i \in \N}|a_i| = \norm{\! \seq{a_i}{i \in \N} \!}.$$
For the other direction, fix any constant $c < 1$, and pick $N$ large enough that $\sum_{i \leq N}|a_i| > c\sum_{i \in \N}|a_i|$. Define $f: \closure{D} \to [-1,1]$ by setting 
$$f(x) \,=\, \begin{cases}
1 &\text{ if $x = d_i$ for some $i \leq N$ and $a_i \geq 0$,} \\
-1 &\text{ if $x = d_i$ for some $i \leq N$ and $a_i < 0$,} \\
0 &\text{ otherwise.}
\end{cases}$$
This function is continuous on $\closure{D}$, because it is nonzero at only finitely many isolated points in $\closure{D}$. By the Tietze Extension Theorem, there is a continuous function $\hat f: K \to [-1,1]$ such that $\hat f \rest \closure{D} = f$. Thus $\hat f \in C(K)$ and $\norm{\hat f} = 1$, and 
$$\textstyle \Big( \sum_{i \in \N}a_i\be_i \Big)\big( \hat f \,\big) \,=\, \sum_{i \leq N} a_i f(d_i) \,=\, \sum_{i \leq N}|a_i| \,>\, c\sum_{i \in \N}|a_i|.$$
Because $\norm{\hat f} = 1$, it follows that 
$$\textstyle \norm{\phi(\seq{a_i}{i \in \N})} \,=\, \bignorm{\sum_{i \in \N}a_i\be_i} \,>\, c\sum_{i \in \N}|a_i| \,=\, c\norm{\! \seq{a_i}{i \in \N} \!}.$$  
Because $c < 1$ was arbitrary, this shows $\norm{\phi(\seq{a_i}{i \in \N})} \geq \norm{\! \seq{a_i}{i \in \N} \!}$.
\end{proof}

A sequence of measure $\seq{\mu_m}{m \in \w}$ on a Boolean algebra $\AA$ is \emph{elementwise bounded} if $\sup_{m \in \w}|\mu_m(a)| < \infty$ for all $a \in \AA$, and it is \emph{uniformly bounded} if $\sup_{m \in \w}\norm{\mu_m} < \infty$. 
We say that $\AA$ has the \emph{Nikodym property} if every elementwise bounded sequence of measures on $\AA$ is uniformly bounded. 

\begin{theorem}\label{thm:Nik}
Every infinite Boolean algebra with the Nikodym property has size at least $\mathrm{cov}(\NN)$. 
\end{theorem}
\begin{proof} 
Let $\AA$ be an infinite Boolean algebra, and let $\mathrm{St}(\AA)$ denote the Stone space of $\AA$. 
As Schachermayer observed in \cite{Sch}, $\AA$ has the Nikodym property if and only if the 
subspace $Z = \mathrm{Span}\set{\chi_{[a]}}{a \in \AA}$ of $C(\mathrm{St}(\AA))$ is barrelled. 
Here we denote by $[a]$ the clopen subspace of $\mathrm{St}(\AA)$ corresponding to an element $a \in \AA$. 
In particular, $\AA$ has the Nikodym property only if $C(\mathrm{St}(\AA)$ has a barrelled subspace with dimension $|\AA|$. 
But Theorem~\ref{thm:RandomWalks}, combined with Lemma~\ref{lem:CK}, implies that every barrelled subspace of $C(\mathrm{St}(\AA))$ has dimension at least $\mathrm{cov}(\NN)$. 
Thus if $\AA$ has the Nikodym property, $|\AA| \geq \mathrm{cov}(\NN)$. 
\end{proof}

We now move on to proving the second case of this sections main theorem: any Banach space whose dual contains $c_0$ or $\ell^p$ for some $p > 1$ has no barrelled subspace with dimension $<\mathrm{cov}(\NN)$. 
In fact, we prove this for the ostensibly broader class of Banach spaces having the following general property. 
Let us say a Banach space $X$ has property $(\dagger)$ if:
\begin{itemize}
\item[$(\dagger)$] There is a sequence $\seq{\be_i}{i \in \w}$ of unit functionals in $X'$, and a function $f: \w \to \w$, such that for every $\bx \in X$ with $\norm{\bx} = 1$, 
$$\big|\! \set{i \in \w}{|\be_i(\bx)| > \nicefrac{1}{m}} \!\big| \,<\, f(m)$$
for all sufficiently large values of $m$.
\end{itemize}
Before going on to how this property relates to barrelledness, let us first show that if $X'$ contains $c_0$ or $\ell^p$ for any $p > 1$, then $X$ has property $(\dagger)$.

\begin{proposition}\label{prop:containslp}
If $X$ is a Banach space whose dual space contains either $c_0$ or $\ell^p$ for some $p > 1$, then $X$ has property $(\dagger)$. 
\end{proposition}
\begin{proof}
Let us consider first the case where $X'$ contains $\ell^p$. 
Fix $p > 1$, let $\phi: \ell^p \to X'$ be an injective bounded linear map, and fix $b > 0$ such that $\frac{1}{b}\norm{\bx} \leq \norm{\phi(\bx)} \leq b\norm{\bx}$ for all $\bx \in \ell^p$. 
For each $i \in \w$, let $\be_i^0$ be the sequence in $\ell^p$ with a $1$ in coordinate $i$ and a $0$ in every other coordinate, and let $\be_i = \frac{1}{\norm{\phi(\be^0_i)}} \phi(\be^0_i)$. 

Clearly $\seq{\be_i}{i \in \w}$ is a sequence of unit vectors in $X'$. 
Fix $\bx \in X$ with $\norm{\bx} = 1$, and define $\varsigma: \N \to \{-1,1\}$ by setting 
$\varsigma(i) = 1$ if $\be_i(\bx) \geq 0$ and $\varsigma(i) = -1$ if $\be_i(\bx) < 0$ (i.e., $\varsigma(i)$ indicates the sign of $\be_i(\bx)$). Given $m \in \w$, let 
$$F_m \,=\, \set{i \in \w}{|\be_i(\bx)| > \nicefrac{1}{m}},$$
and observe that 
\begin{align*}
\textstyle \frac{1}{m}|F_m| &\,<\, \textstyle \sum_{i \in F_m}|\be_i(\bx)| \,=\, \sum_{i \in F_m} \varsigma(i)\be_i(\bx) \,\leq\, \bignorm{ \sum_{i \in F_m}\varsigma(i)\be_i} \\
&\,\leq\, \textstyle \bignorm{ \sum_{i \in F_m}\frac{1}{\norm{\phi(\be_i^0)}}\varsigma(i)\phi(\be_i^0)} \,\leq\, b \bignorm{ \sum_{i \in F_m}\varsigma(i)\phi(\be_i^0)} \\ 
&\,\leq\, \textstyle b \bignorm{ \phi\big( \sum_{i \in F_m}\varsigma(i)\be_i^0 \big)} \,\leq\, \textstyle b^2 \bignorm{ \sum_{i \in F_m}\varsigma(i)\be_i^0 } \,=\, b^2|F_m|^{\nicefrac{1}{p}}.
\end{align*}
Consequently, $|F_m|^{1-\frac{1}{p}} < mb^2$ or, setting $\frac{1}{q} = 1 - \frac{1}{p}$ as usual, $|F_m| < m^q b^{2q}$. In particular, taking $f(m)$ to be any integer $> \! m^q b^{2q}$ for every $m$, $f$ witnesses $(\dagger)$.

The argument for $c_0$ is almost exactly the same, but with one simplification. In the final equality in the displayed formulas in the second paragraph, we have $b^2 \bignorm{ \sum_{i \in F_m}\varsigma(i)\be_i^0 } = b^2$ rather than $b^2 \bignorm{ \sum_{i \in F_m}\varsigma(i)\be_i^0 } = b^2|F_m|^{\nicefrac{1}{p}}$. Consequently, we get $|F_m| < mb^2$, and thus taking $f(m)$ to be any integer $> \! m b^2$ for every $m$ gives a witness to $(\dagger)$ in this case.
\end{proof}

\begin{lemma}\label{lem:dagger}
Suppose $X$ is a Banach space with property $(\dagger)$. Then there is a sequence $\seq{\be_i}{i \in \w}$ of unit functionals in $X'$, and there is a partition of $\w$ into finite intervals $I_1,I_2,I_3,\dots$, each of size at least $2$, such that 
$$\sum_{m=1}^\infty \frac{\big|\! \set{i \in I_m}{|\be_i(\bx)| \geq \nicefrac{1}{m}} \!\big|}{|I_m|} \,<\, \infty$$
for every $\bx \in X$ with $\norm{\bx} = 1$. 
\end{lemma}
\begin{proof}
Let $X$ be a Banach space with property $(\dagger)$, and fix a sequence $\seq{\be_i}{i \in \w}$ of unit functionals in $X'$, and a function $f: \w \to \w$, witnessing $(\dagger)$ for $X$. 
Observe that if $g$ is any function $\w \to \w$ strictly larger than $f$, then $\seq{\be_i}{i \in \w}$ and $g$ also witness $(\dagger)$ for $X$ (i.e., increasing $f$ does no harm). 
Thus, without loss of generality, we may (and do) assume that $f$ is non-decreasing, and that $f(m) \geq 2$ for all $m$.

Let $I_1,I_2,I_3,\dots$ be a partition of $\N$ into adjacent intervals (i.e., $\min I_1 = 1$ and $\max I_m = \min I_{m+1}-1$ for all $m$) such that $|I_m| > m^2f(m)$ for every $m$. 
By $(\dagger)$, if $\bx \in X$ and $\norm{\bx} = 1$ then
$$|\set{i \in I_m}{|\be_i(\bx)| \geq \nicefrac{1}{m}}| \,\leq\, |\set{i \in \w}{|\be_i(\bx)| \geq \nicefrac{1}{m}}| \,<\, f(m)$$
for all sufficiently large values of $m$. 
Consequently, $\frac{ | \set{i \in I_m}{|\be_i(\bx)| \geq \nicefrac{1}{m}} | }{|I_m|} < \nicefrac{1}{m^2}$ for all sufficiently large values of $m$. 
Because $\sum_{m=1}^\infty \nicefrac{1}{m^2} < \infty$, this implies that $\sum_{m=1}^\infty \frac{ | \set{i \in I_m}{|\be_i(\bx)| \geq \nicefrac{1}{m}} | }{|I_m|} < \infty$ as well. 
\end{proof}

\begin{theorem}\label{thm:Random0}
Suppose $X$ is a Banach space with property $(\dagger)$. 
If $Y \sub X$ detects the norm-unboundedness of all functional sequences, then $|Y| \geq \mathrm{cov}(\NN)$. 
\end{theorem}
\begin{proof}
Suppose $X$ is a Banach space with property $(\dagger)$. 
Fix a sequence $\seq{\be_i}{i \in \w}$ of functionals in $X'$, and a sequence $I_1,I_2,I_3,\dots$ of intervals partitioning $\w$, that witness the conclusion of Lemma~\ref{lem:dagger} for $X$. 

The basic idea of the proof is the same as for Theorem~\ref{thm:Random1}: we must define a probability space $\Omega$, and an injective mapping from $\Omega$ to norm-unbounded sequences in $X'$, in order to make sense of the notion of a ``random'' norm-unbounded sequence.  

Let $\Omega = \prod_{m \in \w}I_m$, the space of all functions $f: \w \to \w$ such that $f(m) \in I_m$ for all $m$. 
In other words, a point of $\Omega$ selects a single member of each $I_m$. To make $\Omega$ a probability space, simply consider these selections to happen ``randomly'' with respect to the uniform distribution on each $I_m$. 
More formally, for each $m$ let $\mu_m$ denote the uniform probability measure on $I_m$, and then $\mu = \bigotimes_{m \in \w}\mu_m$ is the desired probability measure on $\Omega$. 
To emphasize the probabilistic viewpoint, let 
$$P\big( \varphi(x) \big) \,=\, \mu \big(\! \set{f \in \Omega}{\varphi(f)} \!\big)$$
whenever $\varphi(x)$ is some formula/property defining a subset of $\Omega$. 
The relevant properties of $\mu$ are:
\begin{itemize}
\item[$\circ$] If $F$ is a finite subset of $I_m$, then
$$P\big( f(m) \in F \big) \,=\,\textstyle \nicefrac{|F|}{|I_m|}.$$
This follows from our definition of $\mu$, which uses the uniform probability distribution $\mu_m$ on $I_m$.
\item[$\circ$] As a probability measure, $\mu$ is countably (sub)additive. In particular, 
$$P\big( \varphi_m(f) \text{ holds for some $m \in A$} \big) \,\leq\,\textstyle \sum_{m \in A}P\big( \varphi_m(f) \big)$$
whenever $A \sub \N$ and $\varphi_m(x)$ is some formula/propety defining a subset of $\Omega$ for every $m \in A$.
\item[$\circ$] The measure space $(\Omega,\mu)$ is an atomless probability Borel measure on a compact metrizable space. (That the measure $\mu$ is atomless follows from our assumption in Lemma~\ref{lem:dagger} that $|I_m| \geq 2$ for all $m$.) All such measure spaces are Borel-isomorphic to the Lebesgue measure on the unit interval $[0,1]$ \cite[Theorem 17.41]{Kechris}. In particular, it follows that if $\F$ is a family of $\mu$-null subsets of $\Omega$ and $|\F| < \mathrm{cov}(\NN)$, then $\bigcup \F \not= \Omega$.
\end{itemize}

To each $f \in \Omega$ associate a sequence $\seq{\by^f_m}{m \in \w}$ in $B'$ as follows:
$$\by^f_m \,=\, m\be_{f(m)}.$$
Observe that $\norm{\by_m^f} = m\norm{\be_{f(m)}} = m$ for every $f \in \Omega$, hence $\sup_{m \in \w}\norm{\by^f_m}  = \infty$. 
Thus, as promised, the map $f \mapsto \seq{\by^f_m}{m \in \w}$ gives us a way of associating the members of our probability space $\Omega$ to norm-unbounded sequences in $X'$. 

\begin{claim}
For any given $\bx \in \ell^1$, 
$$P \big( \bx \text{ detects the norm-unboundedness of }\seq{\by^f_m}{m \in \w} \!\big) \,=\, 0.$$
\end{claim}

\noindent \emph{Proof of the claim:} 
Fix $\bx \in X$, and for convenience, for each $m \in \w$ let us define $F_m = \set{i \in I_m}{|\be_i(\bx)| \geq \nicefrac{\norm{\bx}}{m}}$. 
By the first property of $\mu$ listed above, 
$$P \big( |\by^f_m(\bx)| \geq \norm{\bx} \big) \,=\, P \big( |\be_{f(m)}(\bx)| \geq \nicefrac{\norm{\bx}}{m} \big) \,=\, \nicefrac{|F_m|}{|I_m|}$$
for all $m \in \w$. 

Letting $\hat \bx = \nicefrac{\bx}{\norm{\bx}}$, observe that 
$F_m = \set{i \in I_m}{|\be_i(\hat \bx)| \geq \nicefrac{1}{m}}$. 
Because $\norm{\hat \bx} = 1$, and because $\seq{\be_i}{i \in \w}$ and $I_1,I_2,I_3,\dots$ witness the conclusion of Lemma~\ref{lem:dagger} for $X$, this implies $\sum_{m \in \w} \nicefrac{|F_m|}{|I_m|} < \infty$.

Let $\e > 0$. 
Because $\sum_{m \in \w} \nicefrac{|F_m|}{|I_m|}$ converges, there is some $N \in \N$ such that $\sum_{m \geq N} \nicefrac{|F_m|}{|I_m|} < \e$. 
Observe that $|\by^f_m(\bx)| \leq \norm{\by^f_m}\norm{\bx} = m\norm{\bx}$ whenever $m < N$. 
Combining this with the observations above and the countable (sub)additivity of our probability measure $\mu$, 
\begin{align*}
\textstyle P \big( \sup_{m \in \w}|\by_m^f(\bx)| = \infty \big) &\,\leq\, P \big( \textstyle \sup_{m \in \w}|\by_m^f(\bx)| > N\norm{x} \big) \\
&\,\leq\,\textstyle P \big( |\by_m^f(\bx)| > \norm{\bx} \text{ for some } m \geq N \big) \\
&\,\leq\, \textstyle \sum_{m \geq N} P \big( |\by_m^f(\bx)| \geq \norm{\bx} \big) \\
&\,=\, \textstyle \sum_{m \geq N} \nicefrac{|F_m|}{|I_m|} \\ &\,<\, \e.
\end{align*}
As $\e$ was arbitrary, this implies $P \big( \sup_{m \in \w}|\by_m^f(\bx)| = \infty \big) = 0$.
\hfill $\dashv$

\vspace{2.5mm}

To finish the proof of the theorem, let $Y \sub X$. Observe that
\begin{align*}
&\set{f \in \Omega}{ \text{$Y$ detects the norm-unboundedness of }\seq{y^f_m}{m \in \w} } \\
& \qquad \,=\, \textstyle \bigcup_{\bx \in Y} \set{f \in \Omega}{ \text{$\bx$ detects the norm-unboundedness of }\seq{y^f_m}{m \in \w} },
\end{align*}
and by the claim above, this is a size-$|Y|$ union of $\mu$-null sets. If $|Y| < \mathrm{cov}(\NN)$, then by the third property of $\mu$ listed above, 
$$\textstyle \bigcup_{\bx \in Y} \set{f \in \Omega}{ \text{$\bx$ detects the norm-unboundedness of }\seq{y^f_m}{m \in \w} } \,\neq\, \Omega.$$ 
In particular, if $|Y| < \mathrm{cov}(\NN)$ then $Y$ does not detect the norm-unboundedness of every sequence of the form $\seq{y^f_m}{m \in \w}$, and so $Y$ does not detect the norm-unboundedness of all functional sequences.
\end{proof}

\begin{corollary}\label{cor:Random0}
If $X$ is a Banach space with property $(\dagger)$, then every barrelled subspace of $X$ has dimension at least $\mathrm{cov}(\NN)$. 
Moreover, it is consistent with \zfc that no Banach space with property $(\dagger)$ has a barrelled subspace of dimension $\bdd$. 
\end{corollary}
\begin{proof}
This is proved in exactly the same way as Theorem~\ref{thm:RandomWalks} and Corollary~\ref{cor:Con1}, but using Theorem~\ref{thm:Random0} in the place of Theorem~\ref{thm:Random1}. 
\end{proof}

The following theorem summarizes the results of this section:

\begin{theorem}\label{thm:MainCor}
Suppose $X$ is a Banach space whose dual contains either $c_0$ or $\ell^p$ for some $p \geq 1$. 
Then every barrelled subspace of $X$ has dimension at least $\mathrm{cov}(\NN)$. 
Moreover, it is consistent with \zfc that no such Banach space has a barrelled subspace of dimension $\bdd$. 
\end{theorem}
\begin{proof}
This follows immediately from Proposition~\ref{prop:containslp} together with Theorem~\ref{thm:RandomWalks} and Corollaries~\ref{cor:Con1}, and \ref{cor:Random0}.
\end{proof}

We end with two questions raised by the results in this section. 
First, observe that our results apply to an ostensibly larger class of Banach spaces than just those whose dual does not contain $c_0$ or $\ell^p$ for any $p \geq 1$. Theorem~\ref{thm:MainCor} really holds for all Banach spaces satisfying the conclusion of Lemma~\ref{lem:containsl1} or having property $(\dagger)$. Using ideas from \cite{Rosenthal}, it seems clear that $X'$ contains $\ell^1$ if and only if $X$ satisfies the conclusion of Lemma~\ref{lem:containsl1}. But it is not clear that property $(\dagger)$ is equivalent to $X'$ containing $c_0$ or $\ell^p$ for some $p > 1$.

\begin{question}\label{q:BadSpace}
Does every Banach space either have property $(\dagger)$ or else have $\ell^1$ contained in its dual? 
\end{question}

A positive answer to this question would mean that the conclusions of Theorem~\ref{thm:MainCor} apply to all Banach spaces, which would imply $\mathrm{cov}(\NN) \leq \barr$, and the consistency of $\bdd < \barr$. However, while we suspect the answer to Question~\ref{q:BadSpace} is negative, we also imagine there could be other roads to proving $\mathrm{cov}(\NN) \leq \barr$.

\begin{question}
Is $\mathrm{cov}(\NN) \leq \barr$? 
\end{question}

\begin{question}
Is $\barr$ equal to any previously studied cardinal characteristic of the continuum, such as $\mathrm{cov}(\NN)$, $\mathrm{non}(\M)$, or $\bdd$? 
\end{question}

\end{document}